\documentclass{amsart}

\usepackage{amsmath,amssymb,amsthm,amsfonts,graphicx,color}
\usepackage{amssymb,geometry}

\usepackage{color, graphics}

\usepackage{graphicx}
\usepackage{xcolor}

\usepackage[colorlinks=true]{hyperref}

\makeatletter
\def\@currentlabel{2.1}\label{e:dispaa}
\def\@currentlabel{2.21}\label{e:dispau}
\def\@currentlabel{2.22}\label{e:dispav}
\def\@currentlabel{2.23}\label{e:dispaw}
\def\@currentlabel{2.24}\label{e:dispax}
\def\theequation{\thesection.\@arabic\c@equation}
\makeatother

\usepackage[colorlinks=true]{hyperref}
\hypersetup{linkcolor=black,citecolor=black,filecolor=black,urlcolor=black} 

\newcommand{\N}{\mathbb{N}}

\newcommand{\R}{\mathbb{R}}

\newcommand{\eps}{\varepsilon}

\renewcommand{\theequation}{\thesection.\arabic{equation}}
\newtheorem{definition}{Definition}[section]
\newtheorem{lemma}[definition]{Lemma}
\newtheorem{theorem}[definition]{Theorem}
\newtheorem{proposition}[definition]{Proposition}
\newtheorem{corollary}[definition]{Corollary}
\newtheorem{remark}[definition]{Remark}
\newcommand{\bremark}{\begin{remark} \em}
\newcommand{\eremark}{\end{remark} }

\newcommand{\beq}{\begin{equation}}
\newcommand{\eeq}{\end{equation}}

\def\R {\mathbb{R}}

\def\N {\mathbb{N}}

\begin{document}

\title[Some rigidity results]{Some rigidity results and asymptotic properties for solutions to semilinear elliptic P.D.E.}

\author{Matteo Rizzi} \address[M. ~Rizzi]{Mathematisches Institut, Justus Liebig Universit\"{a}t, Arndtstrasse 2, 35392, Giessen, Germany.}
\email{mrizzi1988@gmail.com}

\author{Panayotis Smyrnelis} \address[P. ~Smyrnelis]{Department of Mathematics, University of Athens, 11584 Athens, Greece}
\email[P. ~Smyrnelis]{smpanos@math.uoa.gr}

\begin{abstract}
We will present some rigidity results for solutions to semilinear elliptic equations of the form $\Delta u=W'(u)$, where $W$ is a quite general potential with a local minimum and a local maximum. We are particularly interested in Liouville-type theorems and symmetry results, which generalise some known facts about the Cahn-Hilliard equation.


\end{abstract}

\maketitle
\section{Introduction and main results}

Our aim is to revisit some asymptotic properties and rigidity results for solutions $u$ to the semilinear elliptic P.D.E. 
\begin{equation}
\label{scalar}
\Delta u(x)= W'(u(x)), \ u  \in C^2( D),   \ D\subset \R^n, \ n\geq 2,
\end{equation}
with a potential $W\in C^{1,1}_{loc}(\R)$. The domains $D\subset \R^n$ we consider are connected open unbounded sets such that
\begin{equation}\label{inrad}
\text{$\forall R>0$, $D$ contains a ball of radius $R$.} 
\end{equation}
We shall also assume that on $D$ the solution $u$ takes its values in a bounded interval where $W$ is monotone. For instance, one may consider the Cahn-Hilliard equation
\begin{equation}\label{cahnhil}
\Delta u=W'(u)=u^3-u+\delta\qquad\text{in $\R^n$,}
\end{equation}
with $|\delta|<\frac{2}{3\sqrt{3}}$, so that the polynomial $f(t)=t^3-t+\delta$ admits exactly three real roots 
$$z_1(\delta)<-\frac{1}{\sqrt{3}}<z_2(\delta)<\frac{1}{\sqrt{3}}<z_3(\delta),$$ 
and is negative on the interval $(z_2(\delta),z_3(\delta))$. This equation was largely studied in the literature. For example, some particular solutions were constructed in \cite{HK1} and \cite{KR}, while some results about radial and cylindrical symmetry of solutions and Liouville type results can be found in \cite{rizzi}. Our starting point is the following theorem.
\begin{theorem}\label{start}[\cite{rizzi}]
Let $n\ge 2$, $\delta\in(-\frac{2}{3\sqrt{3}},\frac{2}{3\sqrt{3}})$ and let $u_\delta$ be a solution to (\ref{cahnhil}) such that
\begin{equation}
u_\delta>z_2(\delta) \qquad \text{outside a ball $B_R\subset\R^N$.}\label{u_pos_ball} 
\end{equation}
\begin{enumerate}
\item If $\delta\in(-\frac{2}{3\sqrt{3}},0]$, then $u\equiv z_3(\delta)$. \label{const_case}
\item If $\delta\in(0,\frac{2}{3\sqrt{3}})$, then $u_\delta$ is radially symmetric (not necessarily constant).\label{rad_case}
\end{enumerate}
\label{th_global_radial}
\end{theorem}
Similar results in the case $\delta=0$ can be found in \cite{Farina-rig, Farina}.\\

The purpose here is to extend Theorem \ref{th_global_radial} to more general non linearities.\\

The proofs in \cite{rizzi} are based on some known symmetry results (see \cite{FMR}) which rely on the moving planes method. A key tool in these methods is the maximum principle, even for unbounded domains (see \cite{BCN}).\\

If
\begin{equation}\label{incl}
u(D)\subset[a,b] \text{, and } W'<0, \text{ on } [a,b) \text{ (with $a,b\in\R$),}
\end{equation}
it is straightforward by Lemma \ref{lem1} below that
\begin{equation}\label{as1}
\lim_{d(x,\partial D)\to\infty}u(x)=b \text{, and } W'(b)=0.
\end{equation}
Thus, we shall focus on the more involved problem where
\begin{subequations}\label{incl2}
\begin{equation}\label{incl2u}
u(D)\subset(a,b], \end{equation}
\begin{equation}\label{incl2w}  W'(a)=W'(b)=0 \text{, and } W'<0, \text{ on } (a,b) \text{ (with $a,b\in\R$).}
\end{equation}
\end{subequations}
If we assume in addition to \eqref{incl2}, the nondegeneracy condition:
\begin{equation}\label{nondeg}
 \text{$W'(s)\leq -C_0 (s-a)$ on $[a,s_0]$, for some $C_0>0$ and $s_0\in (a,b)$,}
\end{equation}
we can apply comparison arguments of Berestycki, Caffarelli, and Nirenberg \cite[Lemma 3.2]{BCN} to deduce that 
\begin{equation}\label{as2}
d(x,\partial D)> \eta \Rightarrow u(x)\geq a+\epsilon, \text{ for some constants  } \eta, \, \epsilon>0.
\end{equation}
Consequently, the asymptotic property \eqref{as1} follows again from Lemma \ref{lem1}.\\

On the other hand, in the degenerate case where \eqref{nondeg} does not hold, the asymptotic behaviour of the solutions may be more involved. In the case where $D$ is the complement of a ball, we can relax condition \eqref{nondeg} by assuming that 
\begin{equation}\label{nondeg2}
\text{$W'(s)\leq -C_0 (s-a)^{\frac{n}{n-2}}$ on $[a,s_0]$, for some $C_0>0$, and $s_0\in (a,b)$,}
\end{equation}
Under assumption (\ref{nondeg2}), we can still prove the asymptotic property \eqref{as1} for solutions provided \eqref{incl2} holds (cf. Proposition \ref{p1b}). However, in the case of potentials such that $$\lim_{u\to a^+}\frac{W'(u)}{(u-a)^p}=-\lambda$$ 
for some $\lambda>0$ and $p>\frac{n}{n-2}$, radial solutions $u:\R^n\to(a,b)$ of \eqref{scalar}  satisfying 
\begin{equation}\label{as3}
 \lim_{|x|\to\infty}u_0(x)=a
\end{equation} 
may exist in dimensions $n\geq 3$ (cf. Lemma \ref{prop1}). Therefore, condition \eqref{nondeg2} is optimal to derive \eqref{as1}, when $D$ is the complement of a ball. For general domains, condition \eqref{nondeg2} is not sufficient to deduce the asymptotic behaviour of the solution. In Proposition \ref{p2b}, we construct a solution of \eqref{scalar} in a dumbbell shaped domain $D\subset \R^2$, such that $u\approx a$ on the one side of the neck, while $u\approx b$ on the other side.

To sum up these results, we now state

\begin{theorem}\label{th1} Let $W\in C^{1,1}_{loc}(\R)$  be a potential satisfying \eqref{incl2w}.
\begin{itemize}
\item[(i)] Assume $u\in C^2(\R^n)$ is a solution of \eqref{scalar} such that $u(\R^n)\subset[a,b]$. Then, when $n=2$, or $n\geq 3$ and \eqref{nondeg2} holds, we have either $u\equiv a$, or $u \equiv b$. Otherwise (when $n\geq 3$ and \eqref{nondeg2} does not hold), 
we have either
$u\equiv a$, or $u \equiv b$, or\footnote{For instance, let $u_0$ be the radial solution provided by Lemma \ref{prop1}. Then, by taking $u(x_1,\ldots,x_n,x_{n+1})=u_0(x_1,\ldots,x_n)$, we can see that \eqref{bcinv} holds.}
\begin{equation}\label{bcinv}
\begin{cases}
u(\R^n)\subset (a,C_W], \text{ for a constant $C_W\in (a,b)$ depending only on $W$},\\
\liminf_{|x|\to\infty}u(x)=a.
\end{cases}
\end{equation}
\item[(ii)] Assume the domain $D$ satisfies \eqref{inrad}, and $u\in C^2(D)$ is a solution of \eqref{scalar} such that $u(D)\subset(a,b]$. Then, we have $\lim_{d(x,\partial D)\to\infty}u(x)=b$, provided that \eqref{nondeg} holds. 
\end{itemize}
\end{theorem}

Next, we derive some Liouville type results by considering domains $D\subset \R^n$ satisfying the following condition:
\begin{equation}\label{exrad}
\text{the radii of the balls contained in $\R^n\setminus D$ are uniformly bounded by a constant $\Lambda>0$.}
\end{equation}

\begin{theorem}\label{th2}
Let $W\in C^{1,1}_{loc}(\R)$ be a non negative potential satisfying \eqref{incl2w}, and $W(b)=0$. Assume the domain $D$ satisfies  \eqref{exrad}, and $u\in C^2(\R^n)$ is a bounded entire solution of \eqref{scalar} such that $\sup_{\R^n}u=b$, and $u(D)\subset (a,b]$. Then, $u \equiv b$.
\end{theorem}

\begin{remark}\label{modic} Modica \cite{modica} proved that if $W\in C^2(\R)$ is a non negative potential, and $u$ is a bounded solution of \eqref{scalar} in $\R^n$, then the condition $W(u(x_0))=0$ for some $x_0\in \R^n$ implies that $u$ is constant. In the sequel, a new proof of this result which also applies to potentials $W\in C^{1,1}_{loc}(\R)$ was proposed in \cite{caf}. Therefore, the hypothesis $\sup_{\R^n} u = b$ in Theorem \ref{th2} is not very strong, since the condition $u(D) \subset (a, b]$ yields that  either $u < b$ in $\R^n$ or $u \equiv b$, so that $\sup_{\R^n} u \leq b$. 
\end{remark}

Since the linear behaviour of $W'$ near the local maximum (see condition (\ref{nondeg})) implies that $\lim_{d(x,\partial D)\to\infty}u(x)=b$, when $D$ satisfies \eqref{inrad} (cf. Theorem \ref{th1} (ii)), we obtain a first corollary of Theorem \ref{th2}:
\begin{corollary}\label{cor1}
Let $W\in C^{1,1}_{loc}(\R)$ be a non negative potential satisfying \eqref{incl2w}, \eqref{nondeg}  and $W(b)=0$. Assume the domain $D$ satisfies \eqref{inrad} as well as \eqref{exrad}, and $u\in C^2(\R^n)$ is a bounded entire solution of \eqref{scalar} such that $u(D)\subset (a,b]$. Then, $u \equiv b$.
\end{corollary}

Finally, we particularise Corollary \ref{cor1} in the case where $D$ is the complement of a ball.
\begin{corollary}\label{cor2}
Let $W\in C^{1,1}_{loc}(\R)$ be a nonnegative potential satisfying \eqref{incl2w}, and $W(b)=0$. Assume $u\in C^2(\R^n)$ is an entire solution of \eqref{scalar} such that 
\begin{equation}
\label{a<uleb}
u(x)\in (a,b]\qquad\forall\,x\in\R^n\backslash B_R,
\end{equation}
for some $R>0$. Then, $u \equiv b$, provided that $n=2$, or $n\geq 3$ and \eqref{nondeg2} holds.
\end{corollary}
We will prove these results in Section \ref{sec-Liouville}.
\begin{remark}
Corollary \ref{cor2} was established in \cite{rizzi} for entire solutions to the Cahn-Hilliard equation \eqref{cahnhil}. Here we extend this result to general nonlinearities under optimal assumptions. Indeed, the necessity of condition \eqref{nondeg2} (when $n\geq 3$) for Corollary \ref{cor2} to hold, is clear in view of Lemma \ref{prop1}. 
\end{remark}
Other Liouville type results for stable solutions to semilinear PDEs were established in \cite{DF}. Here there is no stability assumption.\\

After that, we will address the issue of radial symmetry. In \cite{GNN}, the authors prove radial symmetry of solutions to fully nonlinear equations of very general form, provided these solutions have a suitable asymptotic polynomial decay at infinity (see Theorem $4$ there). Here we are interested in radial symmetry of solutions to (\ref{scalar}) with $W$ satisfying (\ref{incl2w}), assuming that either $\lim_{|x|\to\infty} u(x)=b$ or $\lim_{|x|\to\infty} u(x)=a$.\\

The case in which $\lim_{|x|\to\infty} u(x)=b$ is easier. The following result is a consequence of \cite[Proposition 1]{GNN}: 
\begin{proposition}
\label{th-limb}
Let $u:\R^n\to\R$ be a solution to $\Delta u=W'(u)$, $n\ge 3$, where $W\in C^{2}(\R^n)$ is a potential fulfilling (\ref{incl2w}) and such that 
\begin{equation}
\label{cond-reg-W}
[b-\delta,b]\ni t\mapsto \frac{W'(t)}{|t-b|^p}\text{ is H\"{o}lder continuous for some } \delta>0, \, p\ge \frac{n+2}{n-2}.
\end{equation}
Assume that $u<b$ in $\R^n$ and $\lim_{|x|\to\infty}u(x)=b$. Then $u$ is radially symmetric. 
\end{proposition}
On the other hand, if $W$ is convex in an interval $(b-\delta,b)$, then the symmetry result follows from  \cite[Theorem 2]{FMR} in any dimension $n\geq 2$. Assumption (\ref{cond-reg-W}) is not required anymore. In view of Proposition \ref{th-limb} and \cite[Theorem 2]{FMR}, we obtain the following generalisation of Theorem \ref{start} (ii): 
\begin{theorem}\label{th-limb2}
Let $W\in C^{2}(\R)$ be a potential such that $W'(t)<0$ for any $t\in(a,b)$, $W'(a)=0$, and $W(t)\geq W(b)$ for any $t>b$. In addition, we suppose that one of the following is true:
\begin{itemize}
\item[(i)] $n\geq 3$, and $W\in C^{6,\alpha}(b-\delta,b+\delta)$, for some $\delta>0$ and $\alpha\in (0,1)$.
\item[(ii)] $n\geq 2$, and $W$ is convex in $(b-\delta,b)$, for some $\delta>0$.
\end{itemize}
 Assume also that $u:\R^n\to\R$ is a solution to $\Delta u=W'(u)$ such that $u(\R^n\backslash B_R)\subset(a,b)$ and $\lim_{|x|\to\infty}u(x)=b$. Then $u<b$ in $\R^n$ and it is radially symmetric. 
\end{theorem}

\begin{remark}
\begin{enumerate}
\item If a potential $W$ satisfies the assumptions of Theorem \ref{th-limb2}, then it has a local minimum at $t=b$, so that $W'(b)=0$. However, this minimum is not required to be a global one. 
\item If $W'(t)>0$ for $t>b$, it follows from the maximum principle that any bounded solutions $u$ of \eqref{scalar} in $\R^n$, satisfies the bound $u\leq b$.
\item Let $u$ be a solution of \eqref{scalar} in $\R^n$. Then, if $n=2$ or $n\geq 3$ and \eqref{nondeg2} holds, the condition $u(\R^n\setminus B_R)\subset (a,b)$, implies that $\lim_{|x|\to\infty}u(x)=b$ in view of Lemmas \ref{subharmonic} and \ref{lem1} (resp. Proposition \ref{p1b}).
\item For the existence of radial solutions satisfying $\lim_{|x|\to\infty} u(x)=b$, we refer to \cite[Theorem 1, Theorem 4]{BL} and \cite[Theorem 1.3]{tan}.
\end{enumerate}
\end{remark}


We will prove these symmetry results in Section \ref{sec-limb}.\\

Assuming again that $W$ satisfies (\ref{incl2w}), the description of entire  solutions to \eqref{scalar} converging to $a$ at infinity is a much more difficult task. In that case, only a few symmetry results are available, under somewhat restrictive hypotheses on the solution and the nonlinearity. Some results can be found in \cite{CGS}, where a monotonicity assumption is required. As a particular case, their results apply to bounded solutions to the Lane-Emden equation $$-\Delta u=|u|^{p-1}u$$
in $\R^n$, for which several Liuoville type results are known (see for example \cite{BP,Farina-06,LMW}). For future purposes, the main difficulty is to remove the monotonicity and convexity assumption about the non-linearity.\\

By Proposition \ref{p1b}, we know that non trivial solutions can exist only if condition (\ref{nondeg2}) is violated. However, the fact that
\begin{equation}\label{asum1}
u(\R^n\backslash B_R)\subset(a,b)
\end{equation}
cannot guarantee a Liouville type result (cf. Lemma \ref{prop1}), or even radial symmetry under the assumption that
\begin{equation}\label{asum2}
\lim_{|x|\to\infty}u(x)=a.
\end{equation}
In section \ref{sec-lima}, we check that the solutions constructed in \cite{DMPP}, provide examples of nonradial solutions to \eqref{scalar}, such that $u(x)-a$ changes sign in a compact set, and \eqref{asum1} as well as \eqref{asum2} hold. It would be interesting to see if a nonradial solution satisfying $u(\R^n)\subset(a,b)$ and $\lim_{|x|\to\infty}u(x)=a$, may also exist for a potential $W$ having a negative derivative on the range of $u$. To the best of our knowledge, this is a difficult open problem.

\section{Asymptotic behaviour and Liouville type results}\label{sec-Liouville}

We first prove a basic lemma on the asymptotic behaviour of solutions satisfying \eqref{incl}.

\begin{lemma}\label{lem1}
Let $D\subset\R^n$ be a domain satisfying \eqref{inrad}, and let $u$ be a solution of \eqref{scalar} ($W\in C^{1,\alpha}_{loc}(\R)$, $\alpha\in(0,1)$). Assume also that $u(D)\subset[a,b]$, and $W'<0$ on the interval $[a,b)$ (with $a,b\in\R$). Then, $\lim_{d(x,\partial D)\to\infty}u(x)=b$, and $W'(b)=0$ hold. If in addition $D=\R^n$, then we have $u\equiv b$.
\end{lemma}

\begin{proof}
We first recall that for fixed $R>0$, the solution $u$ is uniformly bounded in $C^{2,\alpha}$ (for some $\alpha\in (0,1)$) on the balls $B_R(x)$ satisfying $d(x,\partial D)>R+1$ with $x\in D$. Let $l:=\liminf_{d(x,\partial D)\to\infty} u(x)$, and let  $\{x_k\}\subset D$ be a sequence such that $\lim_{k\to\infty} d (x_k,\partial D)=\infty$, and $\lim_{k\to\infty} u(x_k)=l$. We set $v_k(y)=u(x_k+y)$. In view of the previous estimates, we can apply the Ascoli theorem via a diagonal argument to the sequence $\{v_k\}$, and deduce that up to subsequence, $v_k$ converges in $C^2_{\mathrm{loc}}(\R^n)$ to an entire solution $v_\infty$ of (\ref{scalar}). Moreover, we have
$$v_\infty(0)=l=\min_{y\in \R^n}v_\infty(y),$$
and 
$$0\leq \Delta v_\infty (0)=W'(l) \leq 0,$$
so that, $l=b$, $W'(b)=0$, and $v_\infty\equiv b$. This proves that $\lim_{d(x,\partial D)\to\infty}u(x)=b$, and $W'(b)=0$ hold.\\ 

In the particular case where $D=\R^n$, we have $u\equiv b$, since otherwise $u$ would attain its minimum at a point $x_0$ where
$0\leq \Delta u(x_0)=W'(u(x_0)) <0$, which is a contradiction.
\end{proof}

Next, given a potential satisfying \eqref{incl2w}, we study the existence of solutions such that $u(\R^n)\subset (a,b)$, for $n \geq 3$. The answer to this question depends on the growth of $W'$ in a right neighbourhood of $a$. In Proposition \ref{p1b} below, we first examine the case of potentials for which \eqref{nondeg2} holds.

\begin{proposition}\label{p1b}
Let $n\geq 3$, let $B_\rho\subset\R^n$ be the open ball of radius $\rho$ centred at the origin, and let $W\in C^{1,1}_{loc}(\R)$ be a potential fulfilling \eqref{incl2w}, and \eqref{nondeg2}. Then, every solution 
$u \in C^2(\R^n\setminus B_\rho)$ to \eqref{scalar} such that $u(\R^n\setminus B_\rho)\subset (a,b)$, satisfies $\lim_{|x|\to\infty}u(x)=b$. 
\end{proposition}

\begin{proof}
Without loss of generality, we may assume that $a=0$. Assume by contradiction that 
\begin{equation}\label{hyp11}
u(\R^n\setminus B_\rho)\subset (0,b-\eta], \text{ for some $\eta>0$ small}.
\end{equation}
Then, we have
\begin{equation}\label{assump1}
W'(u)\leq -c_1 u^{\frac{n}{n-2}}, \ \forall u \in [0,b-\eta]
\end{equation}
for a constant $0<c_1<C_0$. We first examine the case where $u$ is radial, that is, $u(x)=v(|x|)$. As a consequence, $v$ solves
\begin{equation}\label{eqrad}
v''(r)+\frac{n-1}{r}v'(r)=W'(v(r)), \ \forall r\in [\rho,\infty).
\end{equation}
Our claim is that $v'(\rho_0)\leq 0$ holds for some $\rho_0\geq\rho$. Indeed, otherwise, we would have 
\begin{equation*}
\forall r \in[\rho,\infty): \ v'(r)> 0, \text{ and } v''(r)\leq  \kappa:=\max_{[v(\rho),b-\eta]}W'<0,
\end{equation*}
which is impossible.
So far, we have proved that $v'(\rho_0)\leq 0$ for some $\rho_0\geq\rho$. By noticing that $v'(\rho_0)=0\Rightarrow v''(\rho_0)<0$ in view of \eqref{eqrad}, one can see that $v'<0$ holds on an interval $(\rho_0,\rho_0+\epsilon)$, for small $\epsilon>0$. Let $l:=\sup\{r>\rho_0: v'<0 \text{ on } (\rho_0,r) \}$. It is clear that $l=\infty$, since otherwise we would deduce that $v'(l)=0$ and $v''(l)<0$, which is a contradiction. This establishes that $v'<0$ on $(\rho_0,\infty)$. Now, it follows from \eqref{eqrad} that
\begin{align*}
\forall r>\rho_1: \ r^{n-1}v'(r)&\leq r^{n-1}v'(r)-\rho_0^{n-1}v'(\rho_0) =\int_{\rho_0}^r s^{n-1}W'(v(s)) ds \\
&\leq -c_1 v^{\frac{n}{n-2}}(r)\int_{\rho_0}^r s^{n-1}ds \leq -k v^{\frac{n}{n-2}}(r) r^n,
\end{align*}
for a constant $k>0$, and for $\rho_1>\rho_0$ large enough. Next, an integration of the previous inequality gives
\begin{align*}
\forall r >\rho_1: v^{-\frac{2}{n-2}}(r)&\geq v^{-\frac{2}{n-2}}(r)- v^{-\frac{2}{n-2}}(\rho_1)\geq \frac{k}{n-2}(r^2- \rho_1^2),
\end{align*}
from which we deduce that 
\begin{equation}\label{contr1}
\text{$v(r)\leq \tilde k r^{2-n}$, for a constant $\tilde k>0$, and for $r>\rho_2>\rho_1$ large enough}.
\end{equation}
 On the other hand, the lower bound provided by Lemma \ref{prop-lower-bound-superharm}:
\begin{equation}\label{contr2}
\text{$\forall r \geq \rho$: $v(r)\ge cr^{2-n}$, for a constant $c>0$,}
\end{equation}
combined with \eqref{assump1} implies that
\begin{align*}
\forall r>\rho_0:  r^{n-1}v'(r)&\leq\int_{\rho_0}^r s^{n-1}W'(v(s)) ds\leq -c_1c^{\frac{n}{n-2}}\int_{\rho_0}^r s^{-1}ds = -c_1c^{\frac{n}{n-2}}\ln\big(\frac{r}{\rho_0}\big).
\end{align*}
As a consequence, we obtain the bound $v'(r)\leq -c_1c^{\frac{n}{n-2}}\ln(\frac{r}{\rho_0})r^{1-n}$, $\forall r>\rho_0$, which contradicts \eqref{contr1}. Therefore the existence of a radial solution satisfying \eqref{hyp11} is ruled out. 

To complete the proof of Proposition \ref{p1b}, we also have to exclude the existence of non radial solutions. Assume by contradiction that $u \in C^2(\R^n\setminus B_\rho)$ is a solution of \eqref{scalar} satisfying \eqref{hyp11}. In view of Lemma \ref{prop-lower-bound-superharm}, $u$ satisfies the lower bound
\begin{equation}\label{subsol}
u(x)> \phi_*(x)= c|x|^{2-n}, \ c>0,
\end{equation}
where $\phi_*$ is a subsolution of \eqref{scalar}, that is, $\Delta\phi_*=0\geq W'(\phi_*)$. Starting from $u$, we shall construct a radial supersolution $\phi^*$ of \eqref{scalar}, such that $\phi_*\le \phi^*$. Let $\rho_{i,m}$ be the rotation of angle $\frac{\pi}{2^m}$ around the $x_i$ coordinate axis of $\R^n$ ($m\geq 1$, $i=1,\ldots,n-1$), and let $$G_m:=\{\rho_{1,m}^{k_1}\circ\ldots\circ \rho_{n-1,m}^{k_{n-1}}: 0\leq k_i\leq 2^{m+1}-1 , \ i=1,\ldots,n-1\}.$$ Using spherical coordinates, one can see that given $|x_0|\geq\rho$, the set $\cup_{m\geq 1} G_m x_0$ is dense in the sphere $\{x\in\R^n: |x|=|x_0|\}$. In particular, we have
\begin{equation}\label{density}
\lim_{m\to\infty}\min_{g\in G_m} u(gx_0)=\min_{|x|=|x_0|} u(x).
\end{equation}
Next, we notice that for every $g\in G_m$, $ x\mapsto u(gx)$ solves \eqref{scalar}. On the other hand, in view of the Kato inequality, $\phi_m(x):=\min_{g\in G_m}u(gx)$ is a supersolution of \eqref{scalar}, satisfying $\phi_*\le\phi_m\leq u$. In addition,  it follows from \eqref{density} that $\phi^*(x):=\lim_{m\to\infty}\phi_m(x)=\min\{u(y): |y|=|x|\}$. Finally, since $|\nabla \phi_m|$ is uniformly bounded on $\R^n\setminus B_\rho$, we obtain that (up to subsequence) $\phi_m$ conververges weakly to $\phi^*$ in $W^{1,2}(B_R\setminus\overline{B_\rho})$, for every $R>\rho$. This implies that $\phi^*$ (which belongs to $W^{1,2}(B_R\setminus\overline{B_\rho})$, for every $R>\rho$) is a radial supersolution of \eqref{scalar} satisfying $\phi_*\le\phi^*\leq u$. To conclude, we deduce from the method of sub- and supersolutions (cf. Section \ref{sec:A1}, and for instance \cite[Lemma 1.1.1]{Dup}), the existence of a radial solution $v\in C^2(\R^n\setminus B_\rho)$, satisfying $0<\phi_*\leq v\leq \phi^*\leq b-\eta$. In view of the first part of the proof, this is a contradiction.

So far, we have established that every solution $u \in C^2(\R^n\setminus B_\rho)$ to \eqref{scalar} such that $u(\R^n\setminus B_\rho)\subset (0,b)$, satisfies $\sup_{\R^n\setminus B_\rho}u=b$. That is, 
\begin{equation}\label{second}
\exists\{x_k\}_{k\in\N}: \ \lim_{k\to\infty} |x_k|=\infty, \text{ and } \lim_{k\to\infty}u(x_k)= b.
\end{equation}

Setting $v_k(y):=u(x_k+y)$, and proceeding as in Lemma \ref{lem1}, we obtain that (up to subsequence) $v_k$ converges in $C^2_{\mathrm{loc}}(\R^n)$ to an entire solution $v_\infty$ of (\ref{scalar}). Furthermore, since $v_\infty(0)=b$, the maximum principle implies that $v_\infty\equiv b$. At this stage we consider 
a minimizer $\phi_R\in H^1(B_R(0))$ of the energy functional
\begin{subequations}\label{subp} 
\begin{equation}
\tilde E(v)=\int_{B_R(0)}\Big(\frac{1}{2}|\nabla v(x)|^2+\tilde W(v(x))\Big)d x,
\end{equation}
in $H^1_0(B_R(0))$, where
\begin{equation}
\tilde W(v)=\begin{cases}
W(a) &\text{ for } v\leq 0\\
W(v) &\text{ for } 0 \leq v\leq b\\
W(b) &\text{ for } v\geq b.
\end{cases}
\end{equation}
\end{subequations}
It is known that $\phi_R$ is a smooth radial solution of \eqref{scalar} in $B_R(0)$, such that $0\leq \phi_R\leq \max_{B_R(0)}\phi_R:= b-\delta_R$ on $B_R(0)$, for some $\delta_R>0$. In addition, we have $\lim_{R\to\infty} \delta_R=0$. Thus, given $\epsilon>0$, we can ensure that 
\begin{itemize}
\item $\delta_R<\epsilon$ for some $R>0$ large enough,
\item and $\phi_R\leq b-\delta_R\leq v_k$ holds on $B_R(0)$, for $k\geq k_R$ large enough.
\end{itemize}
Finally, by applying the sliding method of  Berestycki, Caffarelli, and Nirenberg \cite[Lemma 3.1]{BCN}, we deduce that $u(x)\geq \phi_R(0)\geq b-\epsilon$, provided that $|x|>\rho+R$. This completes the proof of Proposition \ref{p1b}
\end{proof}

In the subcritical case where $W'(u)\sim -\lambda |u-a|^p$ in a right neighbourhood of $a$,  with $\lambda>0$ and $p\in (\frac{n}{n-2},\frac{n+2}{n-2})$, we shall see in Lemmas \ref{prop1} and \ref{lempp} below, that depending on the potential, there may or may not exist a radial solution such that $u(\R^n)\subset (a,b)$.

\begin{lemma}\label{prop1}
Given any $n\ge 3$, $p>\frac{n}{n-2}$ and $\lambda>0$, there exists a potential $W\in C^2(\R)$ fulfilling \eqref{incl2w}, and a solution $u \in C^\infty(\R^n)$ to \eqref{scalar}, such that
\begin{itemize}
\item[a)] $\lim_{u\to a^+}\frac{W'(u)}{|u-a|^{p}}=-\lambda$,
\item[b)] $u$ is radial and radially decreasing (i.e. $u(x)=\tilde u(|x|)$, for a smooth decreasing function 
 $\tilde u:[0,\infty)\to(a,b)$),
\item[c)] $u(\R^n)\subset (a,b)$, and $\lim_{|x|\to\infty}u(x)=a$,
\item[d)] $W''(u(0))>0$.
\end{itemize}
\end{lemma}
\begin{proof}
Without loss of generality, we may assume that $a=0$. First, we note that the function $v(x)=\big(\frac{2((n-2)p-n)}{\lambda(p-1)^2}\big)^{\frac{1}{p-1}}|x|^{-\frac{2}{p-1}} $ solves the equation
$$\Delta v=-\lambda v^p\qquad\text{in $\R^n\backslash\{0\}$}.$$
Next, in order to eliminate the singularity at the origin, we take a smooth cutoff function $\xi:\R\to[0,1]$ such that
\begin{equation*}\begin{cases}
\text{$\xi=1$ in $[3,\infty)$,}\\
\text{$0<\xi<1$ and $\xi'>0$ in $(2,3)$,}\\
\text{$\xi=0$ in $(-\infty,2]$,}
\end{cases}
\end{equation*}
and we consider a function $\tilde u:(1,\infty)\to\R$ such that 
\begin{equation*}\begin{cases}
\tilde u''(r)=\xi(r)\tilde v''(r)&\forall\,r\in [1,\infty),\\
\tilde u(r)=\tilde v(r)&\forall r\geq3.
\end{cases}
\end{equation*}
where $v(x)=:\tilde{v}(|x|)$. One can see that 
\begin{equation}
\label{Deltau<0}
\tilde u''+\frac{n-1}{r}\tilde u'<0 \qquad\text{in $[1,\infty)$.}
\end{equation}
The latter inequality is clear if $r\ge 3$. In order to prove that (\ref{Deltau<0}) holds in $[1,3)$ too, we note that
\begin{equation}
\begin{aligned}
\tilde u'(r)&=-\int_r^\infty \tilde u''(t)dt=-\int_r^\infty \xi(t)\tilde v''(t)dt\\
&=\xi(r)\tilde v'(r)+\int_r^\infty \xi'(t)\tilde v'(t)dt<\xi(r)\tilde v'(r)\le 0,\qquad\forall\, r\in[1,3),
\end{aligned}
\end{equation}
so that
$$\tilde u''+\frac{n-1}{r}\tilde u'<\xi\Big(\tilde v''+\frac{n-1}{r}\tilde v'\Big)\le 0, \qquad\forall\,r\in[1,3).$$
Now, we extend $\tilde u$ to a smooth even positive function on the whole $\R$, still denoted by $\tilde u$, fulfilling $\tilde u'<0$ in $(0,\infty)$, $\tilde u''<0$ in $[0,1)$, so that $\tilde u''+\frac{n-1}{r}\tilde u'<0$ holds in $[0,\infty)$, $\tilde{u}'''(0)=0$ and $\tilde{u}^{(4)}(0)<0$. This can easily be done if we recall that $\tilde{u}$ is affine and decreasing on $[1,2]$. Since $\tilde u$ is monotone in $[0,\infty)$, then it is invertible in this interval with inverse function $\beta:(0,\tilde u(0)]\to[0,\infty)$. Finally, setting $$\varphi(r):=\tilde u''(r)+\frac{n-1}{r}\tilde u'(r),\qquad\forall\,r>0,$$
and $H(s):=\varphi(\beta(s))$, for $s\in(0,\tilde u(0)]$, one can see that $u(x):=\tilde u(|x|)$ satisfies the equation $\Delta u=H(u)$ in $\R^n$. We also notice that $H(\tilde u(0))=n\tilde u''(0)<0$ and $H'(\tilde u(0))=\frac{(n+2)\tilde u^{(4)}(0)}{3\tilde u''(0)}>0$. Thus, one can find a $C^1$ extension of $H$ to the whole $\R$, still denoted by $H$, such that $H< 0$ in $(0,b)$, for some $b>\tilde u(0)$, and $H(b)=0$. By construction, we have $H(u)=-\lambda u^p$ in $(0,\tilde u(3))$, so that $H(0)=H'(0)=0$. In order to conclude the proof it is enough to define $W$ to be the primitive of $H$.
\end{proof}

\begin{lemma}\label{lempp}
Given any $n\ge 3$, $p\in (\frac{n}{n-2},\frac{n+2}{n-2})$, and $\lambda>0$, there exists a potential $W\in C^2(\R)$ fulfilling \eqref{incl2w} and $\lim_{u\to a^+}\frac{W'(u)}{|u-a|^{p}}=-\lambda$, for which there are no radial solutions $u \in C^2(\R^n)$ of \eqref{scalar} such that $u(\R^n)\subset (a,b)$.
\end{lemma}
\begin{proof}
Without loss of generality, we may assume that $a=0$. We consider the function $H(u)=-\lambda u^p$ on an interval $[0,\beta]$, and  since $p\in (\frac{n}{n-2},\frac{n+2}{n-2})$, we set $\epsilon=\frac{n}{p+1}-\frac{n-2}{2}>0$. One can find a $C^1$ extension of $H$ to the whole $\R$, still denoted by $H$, such that 
\begin{itemize}
\item $H< 0$ in $(0,b)$, and $H(b)=0$, for some $b>\beta$. Let $b=\kappa\beta$, with $\kappa>1$.
\item $H([0,b])=[-\lambda\mu\beta^p,0]$ for some $\mu>1$, such that $\kappa\mu<1+\frac{2\epsilon}{n-2}$.
\end{itemize}
Next, define $W\in C^2(\R)$ to be the primitive of $H$ vanishing at $0$. We claim that
\begin{equation}\label{claim123}
\frac{n-2}{2}W'(u)u-n W(u)>0 \text{ on } (0,b].
\end{equation}
Indeed, we have $\frac{n-2}{2}W'(u)u-n W(u)=\epsilon\lambda u^{p+1}$ on $[0,\beta]$. On the other hand, if $u\in [\beta,b]$, then it follows that
$\frac{n-2}{2}W'(u)u-n W(u)\geq  \frac{n-2}{2}W'(u)u -nW(\beta)\geq (\frac{n}{p+1}-\frac{n-2}{2}\kappa\mu)\lambda \beta^{p+1}>0$. Now that \eqref{claim123} is established, we consider a radial solution $u \in C^2(\R^n)$ of \eqref{scalar} such that $u(\R^n)\subset (0,b)$. Setting $v(|x|)=u(x)$ and proceeding as in the proof of Proposition \ref{p1b}, one can see that $v$ satisfies the standard estimates $v(r)=O(r^{-\frac{2}{p-1}})$, $v'(r)=O(r^{-\frac{p+1}{p-1}})$, and $W(v(r))=O(r^{-\frac{2(p+1)}{p-1}})$. To conclude we use the well-known Pohozaev identity:
\begin{equation}\label{po123}
\int_0^r s^{n-1} (\frac{n-2}{2}W'(v(s))v(s)-n W(v(s))\big)ds=\frac{n-2}{2}r^{n-1}v(r)v'(r)+r^n\big(\frac{|v'(r)|^2}{2}-W(v(r))\big).
\end{equation}
We notice that since $p\in (\frac{n}{n-2},\frac{n+2}{n-2})$, the right hand side of \eqref{po123} goes to $0$, as $r\to\infty$. On the other hand, the left hand side of \eqref{po123} is strictly positive in view of \eqref{claim123}. This rules out the existence of radial solutions such that $u(\R^n)\subset (0,b)$.

\end{proof}

The next Proposition examines the existence of radial solutions in the different regimes. 
\begin{proposition}
Let $n\geq 3$, and let $W\in C^{1,1}_{loc}(\R)$ be a potential satisfying \eqref{incl2w}.
\begin{itemize}
\item[(i)] If \eqref{nondeg2} holds, there are no radial solutions $u \in C^2(\R^n)$ of \eqref{scalar} such that $u(\R^n)\subset (a,b)$.
\item[(ii)] If  $\limsup_{u\to a^+}\frac{|W'(u)|}{|u-a|^{\frac{n+2}{n-2}}}=0$ holds, there exists a radial solution $u \in C^2(\R^n)$ of \eqref{scalar} such that $u(\R^n)\subset (a,b)$. 
\item[(iii)] Otherwise, if neither \eqref{nondeg2} nor $\limsup_{u\to a^+}\frac{|W'(u)|}{|u-a|^{\frac{n+2}{n-2}}}=0$ hold, depending on $W$, there may or may not exist a radial solution $u \in C^2(\R^n)$ of \eqref{scalar} such that $u(\R^n)\subset (a,b)$.
\end{itemize}
\end{proposition}

\begin{proof}
(i) A radial solution $u \in C^2(\R^n)$ of \eqref{scalar} such that $u(\R^n)\subset (a,b)$, decays to $a$, as $|x|\to\infty$. In view of Proposition \ref{p1b}, it is clear that such a solution does not exist when \eqref{nondeg2} holds.

(ii) Now, assume that $\limsup_{u\to a^+}\frac{|W'(u)|}{|u-a|^{\frac{n+2}{n-2}}}=0$ holds, and define
\begin{equation}
\tilde W(v)=\begin{cases}
W(v) &\text{ for }  v\leq b\\
W(b) &\text{ for } v\geq b.
\end{cases}
\end{equation}
Theorem 4 of \cite{BL} provides the existence of a radial solution $u \in C^2(\R^n)$ of $\Delta u=\tilde W'(u)$, such that $u>a$, and $\lim_{|x|\to\infty}u(x)=a$. By the maximum principle, we have $u(\R^n)\subset (a,b)$, and thus $u$ solves $\Delta u=W'(u)$.

Finally, (iii) follows from Lemmas \ref{prop1} and \ref{lempp}.
\end{proof}

As we mentioned in the Introduction, for general domains, condition \eqref{nondeg2} is not sufficient to derive the asymptotic property \eqref{as1} of solutions. Proposition \ref{p2b} below, provides examples of solutions having a different asymptotic behaviour.

\begin{proposition}\label{p2b}
Let $p>1$, and let $W\in C^{1,1}_{loc}(\R)$ be a potential fulfilling \eqref{incl2w}, as well as 
\begin{equation}\label{chk0}
\forall u \in [a,b]: \ W'(u)\geq -c (u-a)^p,  \text{ for a constant $c>0$.}
\end{equation}
Let $D=\{x\in\R^2: |x_2|<\psi(x_1)\}$, where $\psi\in C^\infty(\R)$ is a positive function such that $\psi(s)=\lambda|s|$, for $|s|>\epsilon$ (with $\lambda \, ,\epsilon>0 $ sufficiently small, depending on $W$). Then, there exists a solution $u\in C^2(D)$ of \eqref{scalar} such that $u(D)\subset(a,b)$, and
\begin{equation}\label{doublel}
\lim_{x_1\to+\infty}u(x)=a \text{ and } \lim_{x_1\to-\infty}u(x)=b.
\end{equation}
\end{proposition}

\begin{proof}
Without loss of generality we may assume that $a=0$. We shall first construct a supersolution $\phi^*$ of \eqref{scalar} in $D$. We define the auxilliary functions
\begin{subequations}
\begin{equation}
f(re^{i\theta})=r^{-\frac{2}{p-1}}g(\theta),
\end{equation}
with $g:[-\theta_0,\theta_0]\to(0,\infty)$ ($\theta_0<\frac{\pi}{2}$), a positive solution of the O.D.E.:
\begin{equation}
g''(\theta)=-c g^p(\theta)-\frac{4}{(p-1)^2}g(\theta).
\end{equation}
\end{subequations}
Next, setting $\lambda=\tan(\theta_0)$, one can check that 
\begin{equation}\label{chk1}
\Delta f(x)=-c(f(x))^p \text{ in the sector $S=\{x_1>0, |x_2|<\lambda x_1\}$}.
\end{equation}
In addition, we have $f(x)>b$ in the set $\{0<x_1\leq \epsilon, |x_2|<\lambda x_1\}$, provided that $\epsilon>0$ is sufficiently small.
Finally, we take 
\begin{equation}
\phi^*(x)=\begin{cases}
 \min (f(x),b) &\text{ when } x_1>\epsilon,\text{ and } |x_2|<\lambda x_1.\\
b &\text{ when } x_1\leq\epsilon,\text{ and } |x_2|<\psi(x_1).
\end{cases}
\end{equation}
Using the Kato inequality, one can see that $\phi^*$ is a supersolution of \eqref{scalar} in $D$. Indeed, in view of \eqref{chk0} and  \eqref{chk1}, we have
\begin{equation}\label{chk2}
\Delta \phi^*\leq -cf^p \chi_{\{f<b\}}\leq W'(\phi^*) \text{ in } H^1_{loc}(D),
\end{equation}
where $\chi$ is the characteristic function. 

To construct a subsolution $\phi_*$ of \eqref{scalar} in $D$, we take 
\begin{equation}
\phi_*(x)=\begin{cases}
 e(x_1) &\text{ when } x_1<0,\text{ and } |x_2|<\psi(x_1)\\
0 &\text{ when } x_1\geq 0,\text{ and } |x_2|<\psi(x_1),
\end{cases}
\end{equation}
where $e:(-\infty,0]\to [0,b)$ is the heteroclinic orbit, solving
\begin{equation}
e''(s)=W'(e(s)), \ e(0)=0, \lim_{s\to-\infty}e(s)=b.
\end{equation}
The existence of such a heteroclinic orbit is proved by extending $W$ to an even $C^{1,1}(\R)$ function $\tilde{W}$ such that $\tilde{W}=W$ in $(0,b)$, $\tilde{W}'>0$ in $(b,\infty)$ and considering the phase plane for the ODE $v''=\tilde{W}'(v)$. The situation is analogue to the one we have for the classical double well potential $\frac{1}{4}(1-t^2)^2$.\\ 

It follows again from the Kato inequality that $\Delta \phi_*\geq W'(\phi_*)$ holds in $H^1_{loc}(D)$. In addition, it is clear that $\phi_*< \phi^*$ holds in $D$. Therefore, we deduce from the method of sub- and supersolutions (cf. Section \ref{sec:A1}, and for instance \cite[Lemma 1.1.1]{Dup}), the existence of a solution  $u\in C^2(D)$ of \eqref{scalar} satisfying $\phi_*\leq u \leq \phi^*$. Since $0<u<b$ by the maximum principle, the solution $u$ has all the desired properties.
\end{proof}

Now, we are ready to prove Theorems \ref{th1} and \ref{th2}, and their corollaries.

\begin{proof}[Proof of Theorem \ref{th1}]
(i) Assume $u\in C^2(\R^n)$ is an entire solution of \eqref{scalar} such that $u(\R^n)\subset[a,b]$. When $n=2$, $u$ is a bounded superharmonic function defined on $\R^2$. Thus, $u$ is constant and equal to a critical point of $W$. That is, $u\equiv a$ or $u \equiv b$.\\ 

In higher dimensions $n \geq 3$, we have by the maximum principle either $u\equiv a$, or $a<u\leq b$ on $\R^n$. We shall first assume that $a<u\leq b$ as well as \eqref{nondeg2} hold, and we shall prove that $u\equiv b$. In view of \eqref{nondeg2}, Proposition \ref{p1b} implies that $\lim_{|x|\to\infty} u(x)=b$, and $a+\epsilon\leq u\leq b$ holds on $\R^n$, for some $\epsilon>0$. Thus, $u \equiv b$, by Lemma \ref{lem1}.\\

Next, we consider again an entire solution $u$ of \eqref{scalar} satisfying $u(\R^n)\subset[a,b]$ in dimensions $n\geq 3$, but without assuming \eqref{nondeg2}.
By the maximum principle, we have either $u\equiv a$, or $u \equiv b$, or $a<u<b$. Let
$$\mathcal F=\{\text{$u$ is a solution of \eqref{scalar} such that } u(\R^n)\subset(a,b)\},$$
$$C_W=\sup\{u(x): x\in \R^n, u \in \mathcal F\}.$$
Our first claim is that 
\begin{equation}
\label{infF<b}
C_W<b.
\end{equation}
Indeed, assume by contradiction that there exists a sequence $\{u_k\}\subset\mathcal F$, and a sequence $\{x_k\}\subset\R^n$, such that $\lim_{k\to\infty} u_k(x_k)=b$. Setting $v_k(y)=u_k(x_k+y)$, and proceeding as in Lemma \ref{lem1}, we obtain that (up to subsequence) $v_k$ converges in $C^2_{\mathrm{loc}}(\R^n)$ to an entire solution $v_\infty$ of (\ref{scalar}). Furthermore, since $v_\infty(0)=b$, the maximum principle implies that $v_\infty\equiv b$. At this stage we consider 
the minimizer $\phi_R\in H^1(B_R(0))$ defined in \eqref{subp}. 
It is known that $\phi_R$ is a smooth radial solution of \eqref{scalar} in $B_R(0)$, such that $a\leq \phi_R\leq b-\delta_R$ on $B_R(0)$, for some $\delta_R>0$. In addition, by taking $R>R_0$ large enough, we have $a<\phi_R\leq b-\delta_R$ on $B_R(0)$. Thus, for fixed $R>R_0$, we can ensure that $a<\phi_R\leq b-\delta_R\leq v_k$ holds on $B_R(0)$, provided that $k\geq k_R$ is large enough. Finally, by applying the sliding method of  Berestycki, Caffarelli, and Nirenberg \cite[Lemma 3.1]{BCN}, we deduce that for $k\geq k_R$, $v_k$ as well as $u_k$ are entire solutions of \eqref{scalar} satisfying respectively $v_k(\R^n)\subset [a+\epsilon_R,b]$, and  $u_k(\R^n)\subset [a+\epsilon_R,b]$, with $\epsilon_R:=\phi_R(0)-a>0$. In view of Lemma \ref{lem1}, this implies that $u_k\equiv b$, for $k\geq k_R$, which is a contradiction. This proves (\ref{infF<b}).\\

The fact that $\liminf_{|x|\to\infty}u(x)=a$ holds for every $u\in\mathcal F$ also follows from Lemma \ref{lem1}. Indeed, assuming by contradiction that $\liminf_{|x|\to\infty}u(x)>a$ we would obtain that $u(\R^n)\subset [a+\epsilon,b]$, for some $\epsilon>0$. Therefore, using Lemma \ref{lem1}, we conclude that $u\equiv b$, which is a contradiction.\\

(ii) Now, assume the domain $D$ satisfies \eqref{inrad}, and $u\in C^2(D)$ is a solution of \eqref{scalar} such that $u(D)\subset(a,b]$. In the nondegenerate case where \eqref{nondeg} holds, \cite[Lemma 3.2]{BCN} implies that $a+\epsilon<u(x)\leq b$ holds for some $\epsilon>0$, provided that $d(x,\partial D)>\eta$, for some $\eta>0$. Thus, in view of Lemma \ref{lem1}, we have  $\lim_{d(x,\partial D)\to\infty}u(x)=b$.\\ 

\end{proof}

\begin{proof}[Proof of Theorem \ref{th2}]

On the one hand, since $$\sup_{\R^n}u=b,$$ 
let $\{x_k\}_{k\in\N}\subset\R^n$ be a sequence such that $\lim_{k\to\infty} u (x_k)=b$, and set $v_k(y)=u(x_k+y)$. Proceeding as in the proof of Theorem \ref{th1}, one can see that (up to subsequence), $v_k$ converges in $C^2_{\mathrm{loc}}(\R^n)$ to an entire solution $v_\infty\equiv b$. In particular, given $R>0$ and $\delta>0$, we have $u(x)\in [b-\delta,b]$, provided that $x \in B_R(x_k)$, and $k\geq K(R,\delta)$ is large enough.

On the other hand, let $\iota:=\inf_{\R^n}u\leq b$, and assume by contradiction that $\iota<b$ and $W(\iota)>0$. Next, define the auxiliary potential
\begin{equation*}
\tilde W(u)=\begin{cases}
W(u)&\text{for } u\geq \iota\\
W(\iota)&\text{for } u\leq \iota,
\end{cases}
\end{equation*}
and consider a minimiser $\phi_R\in H^1(B_R(0))$ of the energy functional 
\begin{equation*}
\tilde E(v)=\int_{B_R(0)}\Big(\frac{1}{2}|\nabla v(x)|^2+\tilde W(v(x))\Big)d x,
\end{equation*}
in the class $\mathcal A=\{v\in H^1(B_R(0)), \, v=\iota \text{ on }\partial B_R(0)\}$. Setting $\sigma:=\min\{ t\geq\iota: W(t)=0\}$, and $\sigma_R:=\sup_{B_R(0)} \phi_R$, one can see that $$\iota\leq \phi_R\leq \sigma_R<\sigma$$ 
holds for every $R>0$, since $W(\iota)>0$. In addition, $\phi_R$ is a smooth radial solution of \eqref{scalar} in $B_R(0)$, such that $\lim_{R\to\infty} \sigma_R=\sigma$. Thus by taking $R>0$ large enough, we can ensure that $\iota<\sigma_R<b$. As a consequence, we also have $\phi_R(y)\leq u(y+x_k)$, provided that $y \in B_R(0)$, and $k\geq K(R,\delta)$. Finally, by applying the sliding method of  Berestycki, Caffarelli, and Nirenberg \cite[Lemma 3.1]{BCN}, we deduce that for $u \geq\sigma_R>\iota$, holds on $\R^n$, which is a contradiction.

So far we have established that $W(\iota)=0$, so that $\iota=\sigma$. To complete the proof of Theorem \ref{th2}, it remains to show that $\iota=b$. Indeed, if $\iota<b$ and $W(\iota)=0$, then in particular $\iota<a$, since $W'<0$ on $(a,b)$. Let $\{z_k\}_{k\in\N}\subset\R^n$ be a sequence such that $\lim_{k\to\infty} u (z_k)=\iota$, and set $w_k(y)=u(z_k+y)$. Proceeding as previously, we obtain that (up to subsequence), $w_k$ converges in $C^2_{\mathrm{loc}}(\R^n)$ to an entire solution $w_\infty\equiv \iota$. In particular, given $R_0=\Lambda+1$ (cf. \eqref{exrad}) and $\eta>0$ such that $\iota+\eta<a$, we have $u(x)\in [\iota,\iota+\eta]$, provided that $x \in B_{R_0}(z_k)\cap D$, and $k\geq \tilde K(\eta)$ is large enough (we note that, in view of \eqref{exrad}, $B_{R_0}(z_k)\cap D\ne \emptyset$). This is a contradiction. Therefore, we have proved that $\iota=b$, and $u \equiv b$.

\end{proof}

\begin{proof}[Proof of Corollary \ref{cor1}]
Under the assumptions of Corollary \ref{cor1}, we can apply Theorem \ref{th1} (ii) to deduce that $\lim_{d(x,\partial D)\to \infty}u(x)=b$. On the other hand, in view of Remark \ref{modic} we have $u\leq b$, so that $\sup_{\R^n}u=b$. Therefore, Theorem \ref{th2} implies that $u \equiv b$.
\end{proof}

\begin{proof}[Proof of Corollary \ref{cor2}]
When $n\geq 3$, we first apply Proposition \ref{p1b} to deduce that $\lim_{|x|\to\infty}u(x)=b$. Next, in view of Remark \ref{modic} we obtain that $\sup_{\R^n} u=b$. Finally, Theorem \ref{th2} implies that $u \equiv b$.
On the other hand, when $n=2$, $u$ is superharmonic in $\{x \in\R^2: |x|>R\}$. Setting $\gamma :=\min_{|x|=R+1} u(x)\in (a,b)$, we deduce from Lemma \ref{subharmonic}, that $u(x)\in [\gamma,b)$, provided that $|x|>R+1$. In view of Lemma \ref{lem1}, Remark \ref{modic} and Theorem \ref{th2}, we conclude as previously  that $u\equiv b$.
\end{proof}

\section{Radial symmetry for solutions converging to the local minimum: proofs of Proposition \ref{th-limb} and Theorem \ref{th-limb2}}\label{sec-limb}

In this section we give the proofs of Proposition \ref{th-limb} and Theorem \ref{th-limb2}. 

\begin{proof}[Proof of Proposition \ref{th-limb}] 
First we note that $v:=b-u>0$ is bounded and subharmonic outside $B_R$, in fact $-\Delta v=\Delta u=W'(b-v)\le 0$ in $\R^n\backslash B_R$, hence by \cite[Lemma 22]{FMR} we have the decay estimate 
\begin{equation}
\label{dec-v}
v(x)\le C|x|^{2-n}\qquad\forall,\ |x|\ge \rho.
\end{equation}
Next, it follows from \cite[Proposition 1, Theorem 4]{GNN} that $v$ is radial.
\end{proof}

\begin{proof}[Proof of Theorem \ref{th-limb2}]
First we show that $u<b$ in all $\R^n$. By the strong maximum principle, it is enough to prove that $u\le b$ in $\R^n$. For this purpose, assume by contradiction that $c:=\sup_{\R^n}u=\max_{\R^n}u>b$, and $W(c)>W(b)$. Next, define the auxiliary potential
\begin{equation*}
\tilde W(u)=\begin{cases}
W(b)&\text{for }  u\leq b\\
W(u)&\text{for } b\leq u\leq c\\
W(c)&\text{for } u\geq c,
\end{cases}
\end{equation*}
and consider a minimiser $\phi_R\in H^1(B_R(0))$ of the energy functional 
\begin{equation*}
\tilde E(v)=\int_{B_R(0)}\Big(\frac{1}{2}|\nabla v(x)|^2+\tilde W(v(x))\Big)d x,
\end{equation*}
in the class $\mathcal A=\{v\in H^1(B_R(0)), \, v=c \text{ on }\partial B_R(0)\}$. We know that $\phi_R$ is a radial solution to \eqref{scalar} such that $b<\min_{B_R(0)} \phi_R=\phi_R(0)<c$, for $R\geq R_0$ sufficiently large, since $\phi_R(0)\to b$ as $R\to\infty$. In addition, since $u(\R^n\setminus B_R(0))\subset (a,b)$, we have $u(x+x_0)< \phi_{R_0}(x)$ on $B_{R_0}(0)$, provided that $|x_0|>R+R_0$. Finally, by applying the sliding method of  Berestycki, Caffarelli, and Nirenberg \cite[Lemma 3.1]{BCN}, we deduce that $u \leq\phi_{R_0}(0)<c$ holds on $\R^n$, which is a contradiction.

So far we have established that $W(c)=W(b)$. To conclude that $u\leq b$, it remains to show that $c=b$. Indeed, if $c>b$ and $W(c)=W(b)$, then $c$ is a local minimum of $W$ satisfying $W'(c)=0$, and there exists $x_0\in \R^n$ such that $u(x_0)=c$. Thus, by the maximum principle, we obtain $u\equiv c$, which is excluded.

To complete the proof of Theorem \ref{th-limb2}, we shall use Proposition \ref{th-limb}, \cite[Theorem 2]{FMR}, and the regularity of $W$. We first assume that hypothesis (i) holds, and distinguish the following cases. 

a) If $W''(b)>0$, then $v:=b-u>0$ is a decaying entire solution to $$-\Delta v=f(v):=W'(b-v),$$
therefore it is radial by \cite[Theorem 2]{FMR}, since $f'(t)\le 0$ for $t\in(0,\delta)$. 

Otherwise, $W''(b)=0$ implies that $W'''(b)=0$, since $W\in C^6(\R)$, thus we shall examine the sign of $\frac{d^4  W}{d u^4}(b)$.

b) In the case where $\frac{d^4  W}{d u^4}(b)>0$, the radial symmetry of $u$ follows again from \cite[Theorem 2]{FMR}, since $f'(t)\le 0$ holds for $t\in(0,\delta)$. 

c) In the case where $\frac{d^4  W}{d u^4}(b)=0$, we have $\frac{d^5  W}{d u^4}(b)=0$, and $[b-\delta,b]\ni t\mapsto \frac{W'(t)}{|t-b|^5}$ is H\"{o}lder continuous. Moreover, $5 \ge \frac{n+2}{n-2}$ holds for every $n\ge 3$, hence the result follows from Proposition (\ref{th-limb}). 

Finally, in the case where hypothesis (ii) holds, the result is straightforward in view of \cite[Theorem 2]{FMR}.

\end{proof}

\section{A nonradial solution converging to the local maximum}\label{sec-lima}

In this section we will provide an example of a potential $W$ of the form (\ref{incl2w}) for which equation (\ref{scalar}) admits a solution $u$ such that $u(x)>a$ for $|x|>R$ and $\lim_{|x|\to\infty}u(x)=a$, but $u$ is not radial.\\

The counterexample can be found  in \cite{DMPP} using the Yamabe equation 
\begin{equation}
\label{Jamabe}
-\Delta u=\frac{n(n-2)}{4}|u|^{\frac{4}{n-2}}u \  \text{  in $\R^n$, $n\ge 3$.}
\end{equation}
Equation (\ref{Jamabe}) is variational, in the sense that it is the Euler-Lagrange equation of the energy functional
$$E(u):=\frac{1}{2}\int_{\R^n}|\nabla u|^2-\frac{(n-2)^2}{8}\int_{\R^n}|u|^{\frac{2n}{n-2}}.$$
It is known that the only finite energy positive solutions are given by
$$\mu^{-\frac{n-2}{2}}U(\mu^{-1}(x-\xi)),\qquad U(x):=\left(\frac{2}{1+|x|^2}\right)^{\frac{n-2}{2}},\,\mu>0,\,\xi\in\R^n.$$
These solutions which are called the \textit{standard bubbles}, are also the only positive solutions of \eqref{Jamabe} (see \cite{CGS}).

Using these bubbles, in \cite{DMPP} the authors construct a sequence of bounded entire solutions $\{u_k\}_{k\ge k_0}$ to (\ref{Jamabe}) in $\R^n$ of the form
\begin{equation}
u_k:=v_k+\phi_k,
\end{equation}
where the approximate solution $v_k$ is given by
\begin{equation}
\label{approx-sol}
\begin{aligned}
v_k(x)&:=U(x)-\sum_{j=1}^k \mu_k^{-\frac{n-2}{2}}U(\mu_k^{-1}(x-\xi_{j,k})),\\
\mu_k&=c_n k^{-2}\text{ for $n\ge 4$, $\mu_k=c_3 k^{-2}(\log k)^{-2}$ for $n=3$}\\
\xi_{j,k}&:=(\cos(\frac{2\pi j}{k}),\cos(\frac{2\pi j}{k}),0,\dots,0)\qquad 1\le j\le k
\end{aligned}
\end{equation}
and the corrections $\phi_k$ fulfil
\begin{equation}
\label{correction}
|\phi_k(x)|\le \frac{c}{\log k(1+|x|)}\text{ if $n=3$,}\qquad|\phi_k(x)|\le \frac{c}{k^{\alpha_n}(1+|x|^{n-2})}\text{ if $n\ge 4$, with $\alpha_n>0$.}
\end{equation}
As a consequence, these solutions $v_k$ are $L^\infty(\R^n)$ close to a linear combination of $k+1$ rescaled bubbles. One of them is positive and centred at the origin, the other ones are negative and centred along the unit circle $S^1\subset\R^2$. It particular, they are sign changing solutions. Moreover, it follows from (\ref{approx-sol}) and (\ref{correction}) that
\begin{equation}
u_k(x)\to 0 \text{ as $|x|\to\infty$, for any $k\ge k_0$}.
\end{equation}
We are going to check that $u_k$ is positive outside a ball.
\begin{lemma}
\label{lemma-uk>0}
There exist $\bar{r}>0$ and $\bar{k}>0$ such that $u_k(x)>0$ if $|x|>\bar{r}$ and $k\ge\bar{k}$.
\end{lemma}
\begin{proof}
We will show that, for $k$ large enough, the approximate solution $v_k$ fulfils
\begin{equation}
\label{approx-sol>0}
v_k(x)>\frac{2}{3}U(x)>0\qquad\text{if $|x|>\bar{r}$}
\end{equation}
for some large $\bar{r}>0$. Then we apply (\ref{correction}) to conclude that
$$u_k(x)=v_k(x)+\phi_k(x)>\frac{2}{3}U(x)-\frac{C}{(1+|x|)\log k}>\frac{1}{2}U(x)>0$$
outside a large ball in dimension $n\ge 3$. Similarly, in higher dimension we have
$$u_k(x)=v_k(x)+\phi_k(x)>\frac{2}{3}U(x)-\frac{C}{k^{\alpha_n}(1+|x|^{n-2})}>\frac{1}{2}U(x)>0$$
outside a large ball.\\

In order to prove (\ref{approx-sol>0}), we note that, in dimension $n=3$ we have
\begin{equation}\notag
\begin{aligned}
v_k(x)&=\left(\frac{2}{1+r^2}\right)^{\frac{1}{2}}-\sum_{j=1}^k\mu_k^{-\frac{1}{2}}U\left(\frac{x-\xi_{j,k}}{\mu_k}\right)\ge
\left(\frac{2}{1+r^2}\right)^{\frac{1}{2}}-k\mu_k^{-\frac{1}{2}}
\left(\frac{2\mu_k^2}{\mu_k^2+(r-1)^2}\right)^{\frac{1}{2}}  \\&\ge
\left(\frac{2}{1+r^2}\right)^{\frac{1}{2}}\left(1-k\mu_k^{\frac{1}{2}}\left(\frac{1+r^2}{(r-1)^2}\right)^{\frac{1}{2}}\right)=
\left(\frac{2}{1+r^2}\right)^{\frac{1}{2}}\left(1-\frac{\sqrt{c_3}}{\log k}\left(\frac{1+r^2}{(r-1)^2}\right)^{\frac{1}{2}}\right)
>\frac{2}{3}U(x)
\end{aligned}
\end{equation}
where $r=|x|$ and $k$ are large enough. Similarly, in higher dimension, we have
\begin{equation}\notag
v_k(x)\ge\left(\frac{2}{1+r^2}\right)^{\frac{n-2}{2}}\left(1-\frac{c_n^{\frac{n-2}{2}}}{k^{n-3} }\left(\frac{1+r^2}{(r-1)^2}\right)^{\frac{n-2}{2}}\right)
>\frac{2}{3}U(x).
\end{equation}
for $r$ and $k$ large enough. 
\end{proof}
Finally, we can take $b>\|u_{\bar{k}}\|_{L^\infty(\R^n)}$ and define a $C^1(\R)$ function $f$ such that $f(t)<0$ for any $t\in(0,b)$, $f(t)=-\frac{n(n-2)}{4}|t|^{\frac{4}{n-2}}t$ for $|t|\le \|u_{\bar{k}}\|_{L^\infty(\R^n)}$, and $f(b)=0$. Then $f'(0)=0$ and $u_{\bar{k}}$ is a solution to $\Delta u=f(u)$. Taking $W$ to be a primitive of $f$, we have the required counter example. In fact, we have $0<u_{\bar{k}}(x)<b$ in $D:=\R^n\backslash B_{\bar{r}}$, $u_{\bar{k}}\to 0$ as $|x|\to\infty$ but $u_{\bar{k}}$ is sign changing and not radial.

\section{Appendix}

\subsection{The method of sub- and supersolutions}\label{sec:A1}

Let $\Omega\subset\R^n$ be a open set with Lipschitz boundary, and let $f\in C^\alpha_{loc}(\R)$, for some $\alpha\in (0,1)$. We say that $\underline u \in W^{1,2}(\Omega)$ is a subsolution (respectively $\overline u \in W^{1,2}(\Omega)$ is a supersolution) to
\begin{equation}
\label{eq-sl}
\Delta u=f(u),
\end{equation}
if $\Delta \underline u\ge f(\underline u)$ (respectively $\Delta \overline u\le f(\overline u)$) holds in $\Omega$ in the weak sense.

\begin{proposition}\label{psup}
Let $\underline{u}\le \overline{u}$ be a couple of bounded $W^{1,2}(\Omega)$ sub- and supersolutions to 
\eqref{eq-sl}. Then, there exists a solution $u\in C^2(\Omega)\cap W^{1,2}(\Omega)$ to \eqref{eq-sl}, satisfying $\underline{u}\le u\le \overline{u}$.
\end{proposition}

\begin{proof}
We introduce the nonlinearity
\begin{equation}\label{nonlin2}
g(x,u):=
\begin{cases}
f(\underline{u}(x)) &\text{if } u<\underline{u}(x),\\
f(u) &\text{if }  \underline{u}(x)\le u\le\overline{u}(x),\\
f(\overline{u}(x)) &\text{if } u>\overline{u}(x),
\end{cases}
\end{equation}
and set $G(x,u)=\int_0^u g(x,t) dt$. Next, we establish (exactly as in the proof of  \cite[Lemma 1.1.1]{Dup}), the existence of a minimizer $u$ of the energy functional:
\begin{equation}\label{ener2}
\mathcal{E}(v)=\int_\Omega\big(\frac{1}{2}|\nabla v(x)|^2+G(x,v(x))\big)dx,
\end{equation}
in the class $\mathcal{A}=\underline{u}+W^{1,2}_0(\Omega)$. For the sake of simplicity, we consider in the definition of $\mathcal A$, the boundary condition $v=\underline u$ on $\partial \Omega$. However, we could also set $\mathcal{A}=\phi+W^{1,2}_0(\Omega)$, with any $\phi\in W^{1,2}(\Omega)$ such that $\underline u\leq \phi \leq \overline u$ holds on $\partial \Omega$. By construction, $u$ solves the Euler-Lagrange equation 
\begin{equation}\label{eul2}
\Delta u =g(x,u), \ x \in \Omega.
\end{equation}
 Moreover, it follows from the maximum principle that $\underline{u}\le u\le \overline{u}$ in $\Omega$, which yields that $u$ is actually a $C^2(\Omega)$ solution to \eqref{eq-sl}, satisfying $\underline{u}\le u \le \overline{u}$.
\end{proof}

\begin{remark}\label{rrad} 
If in Proposition \ref{psup}, we consider a domain $\Omega=\{x\in \R^n: \rho_1<|x|<\rho_2  \}$, and a couple $\underline{u}\le \overline{u}$ of bounded \emph{radial} sub- and supersolutions to  
\eqref{eq-sl}, then we obtain the existence of a \emph{radial} solution $u\in C^2(\Omega)\cap W^{1,2}(\Omega)$ to \eqref{eq-sl}, satisfying $\underline{u}\le u\le \overline{u}$. Indeed, since the nonlinearity \eqref{nonlin2} and the energy functional \eqref{ener2} are invariant by the orthogonal group $O(n)$, we can look for a minimizer $u$ in the class $\mathcal{A}_{O(n)}=\{v \in \mathcal{A}: v(\sigma x)=v(x), \forall \sigma \in O(n)\}$. By the principle of symmetric criticality \cite{palais}, $u$ is a smooth \emph{radial} solution to \eqref{eul2}, and the bounds $\underline{u}\le u\le \overline{u}$ follow as previously from the maximum principle.
\end{remark}

The method of sub- and supersolutions is also applicable in unbounded domains. In Proposition \ref{p1b}, we apply it in $\Omega=\R^n\setminus B_\rho$, with a radial subsolution $\phi_*(x)= c|x|^{2-n}$, and a radial supersolution $\phi^*\geq\phi_*$, $\phi^*\in W^{1,2}(B_R\setminus \overline B_\rho)$, $\forall R>\rho$. As a consequence of Proposition \ref{psup} and Remark \ref{rrad}, we obtain for every $R>\rho$, a radial solution $v_R$ to \eqref{scalar} in $\Omega_R:=B_R\setminus \overline B_\rho$, satisfying
\begin{itemize}
\item $\phi_*\leq v_R\leq \phi^*$ in $\Omega_R$, 
\item $v_R=\phi_*$ on $\partial \Omega_R$.
\end{itemize}
 In addition, since for any $\alpha\in(0,1)$, the $C^{1,\alpha}$ norm of $\partial \Omega_R$ is uniformly bounded, and the $C^{1,\alpha}$ norm of $\phi_*$ is also bounded in $\overline\Omega$, we deduce that the $C^{1,\alpha}$ norm of $v_R$ is uniformly bounded in $\overline{\Omega_R}$, $\forall R>\rho$ (cf. \cite[Theorem 8.33]{1987130}). Finally, we use the Theorem of Ascoli, via a diagonal argument, to prove that the limit $v=\lim_{R\to\infty} v_R$ exists (up to subsequence) and is a radial solution to \eqref{scalar} in $\Omega$, satisfying $\phi_*\leq v\leq \phi^*$ in $\Omega$.

In Proposition \ref{p2b}, we have a second application of the method of sub- and supersolutions in an unbounded domain $D$, such that $\partial D$ is bounded for the $C^{1,\alpha}$ norm. Here again, we consider an increasing sequence of bounded domains $D_k$, such that $D=\cup_k D_k$, and the boundaries $\partial D_k$ are uniformly bounded for the $C^{1,\alpha}$ norm. In view of Proposition \ref{psup}, we obtain in each domain $D_k$ a solution $u_k$ of \eqref{scalar}, and then by taking the limit $u=\lim_{k\to\infty} u_k$ via the same diagonal argument, we construct the solution $u$ in the whole domain $D$.

\subsection{Two lemmas for superharmonic functions}

Here we recall two classical results on superharmonic functions.
\begin{lemma}
\label{prop-lower-bound-superharm}
Let $n\geq 3$, let $B_\rho\subset\R^n$ be the open ball of radius $\rho$ centered at the origin, and let $u \in C^2(\R^n\setminus B_\rho)$ be a positive and bounded function, such that $\Delta u\le 0$ in $\R^n\backslash B_\rho$. Then, there exists a constant $c>0$ such that $u(x)\ge c|x|^{2-n}$, for any $x\in\R^n\backslash B_\rho$. 
\end{lemma}
\begin{proof}
We fix $y\in\R^n\backslash \overline{B_\rho(0)}$, $\eps>0$ and we prove that $u(y)\ge c|y|^{2-n}-\eps$, for some constant $c>0$ independent of $\eps$, so that the result follows by letting $\eps\to 0$.\\
In order to do so, we note that $$u(x)\ge \inf_{\partial B_\rho}u=:c\rho^{2-n}=c|x|^{2-n}>c|x|^{2-n}-\eps\qquad\forall\,x\in\partial B_\rho.$$
Moreover, taking $R>|y|$ large enough, we have $$c|x|^{2-n}-\eps<0<u(x)\qquad\forall\, x\in\partial B_R.$$
As a consequence, using that $c|x|^{2-n}-\eps$ is harmonic in the set $A:=\{x\in\R^n:\,\rho<|x|<R\}$, the maximum principle yields that $u\ge c|x|^{2-n}-\eps$ in $A$. In particular we have $u(y)\ge c|y|^{2-n}-\eps$.
\end{proof}

\begin{lemma}\label{subharmonic}
Let $B_r(0)\subset \R^2$ be the open ball of radius $r$ centred at the origin, and let $\psi\in C(\R^2\setminus B_r(0))$ be a function such that
\begin{itemize}
\item $\psi \in W^{1,2}_{loc}(\R^2\setminus \overline{B_r(0)})$,
\item $\psi$ is bounded from below on $\R^2\setminus B_r(0)$,
\item $\Delta\psi\leq 0$, on $\R^2\setminus\overline{ B_r(0) }$.
\end{itemize}
Then, $\psi$ attains its minimum on $\partial B_r(0)$.\end{lemma}
\begin{proof}
Let $x_0\in \partial B_r(0)$ be such that $\min_{\partial B_r(0)}\psi=\psi(x_0)$. For every $\epsilon>0$ fixed, we consider the function $\zeta_\epsilon(x)=\psi(x)+\epsilon \ln(|x|/r)$ which is superharmonic on $\R^2\setminus \overline{B_r(0)}$. In addition, 
we have $\zeta_\epsilon(x)> \zeta_\epsilon(x_0)=\psi(x_0)$, provided that $|x|\geq R_\epsilon$ (with $R_\epsilon$ sufficiently large). Thus, by the maximum principle, the minimum of $\zeta_\epsilon$ in the annuli $r\leq |x|\leq R$, with $R\geq R_\epsilon$, is attained at $x_0$. 
This implies, that for every $\epsilon>0$, and $ x \in\R^2\setminus B_r(0) $, we have $\zeta_\epsilon(x)\geq\psi (x_0)\Leftrightarrow \psi(x)\geq \psi (x_0)-\epsilon \ln (|x|/r)$. Finally, letting $\epsilon\to 0$, we obtain that $\psi(x)\geq \psi (x_0)$ holds for every $ x \in\R^2\setminus B_r(0) $.
\end{proof}

\section*{Acknowledgements}
M. Rizzi was partially supported by Justus Liebig University. The authors are particularly grateful to prof. Alberto Farina for his precious remarks and comments.

\end{document}